\documentclass[12pt, a4paper]{article}
\usepackage{amsmath, amsfonts, amssymb, amsthm}
\usepackage{epsfig}
\usepackage{hyperref}
\usepackage{color}

\newlength{\hchng}
\newlength{\vchng}
\newtheorem{thm}{Theorem}[section]
\newtheorem{prop}[thm]{Proposition}

\newtheorem{lemma}[thm]{Lemma}

\newtheorem{preremark}[thm]{Remark}
\newenvironment{remark}{\begin{preremark}\rm}{\medskip \end{preremark}}
\numberwithin{equation}{section}

\newcommand{\re}[1]{(\ref{#1})}
\newcommand{\begeqa}{\begin{eqnarray}}
\newcommand{\eneqa}{\end{eqnarray}}
\newcommand{\begeqaet}{\begin{eqnarray*}}
\newcommand{\eneqaet}{\end{eqnarray*}}
\newcommand{\beeq}{\begin{equation}}
\newcommand{\eeq}{\end{equation}}
\newcommand{\beeqs}{\begin{equation*}}
\newcommand{\eeqs}{\end{equation*}}

\newcommand{\norm}[1]{\left\Vert#1\right\Vert}
\newcommand{\abs}[1]{\left\vert#1\right\vert}

\newcommand{\R}{\mathbb R}
\newcommand{\eps}{\varepsilon}

\newcommand{\grad} {\nabla}

\newcommand{\dist} {\mathrm{dist}}

\DeclareMathOperator*{\osc}{osc}

\title{Boundary regularity for viscosity solutions of fully nonlinear elliptic equations}
\author{Luis Silvestre and Boyan Sirakov}
\begin{document}
\maketitle

\begin{abstract} We provide regularity results at  the boundary for continuous viscosity solutions to nonconvex fully nonlinear uniformly elliptic equations and inequalities in Euclidian domains. We show that (i) any solution of two sided inequalities with Pucci extremal operators is $C^{1,\alpha}$ on the boundary; (ii) the solution of the Dirichlet problem for fully nonlinear uniformly elliptic equations is $C^{2,\alpha}$ on the boundary; (iii) corresponding asymptotic expansions hold. This is an extension to viscosity solutions of the classical Krylov estimates for smooth solutions. 
\end{abstract}

\section{Introduction}

In this work we study the boundary regularity of continuous viscosity solutions of fully nonlinear elliptic equations and inequalities such as
\begin{equation}\label{princeq}
(S)\qquad F(D^2 u, Du, x) = f(x)
\end{equation}
in a bounded domain $\Omega\subset\R^d$, with a Dirichlet boundary condition on a part of the boundary~$\partial \Omega$. All functions considered in the paper will be assumed continuous in $\overline{\Omega}$.  Standing structure hypotheses on the operator $F$ will be its  uniform ellipticity and Lipschitz continuity in the derivatives of $u$:
\begin{itemize}
 \item[(H1)] there exist numbers $\Lambda\ge\lambda>0$, $K\ge0$,  such that for any $x\in \overline \Omega$, $M,N\in \S_d$, $p,q\in \R^n$,
     \beeq\label{elip}
 \mathcal{M}_{\lambda, \Lambda}^+(M-N) +K|p-q| \ge  F(M,p,x)-F(N,q,x)\ge \mathcal{M}_{\lambda, \Lambda}^-(M-N)-K|p-q|.
     \eeq
 \end{itemize}
We denote with $\mathcal{M}_{\lambda, \Lambda}^\pm(M)$  the extremal Pucci operators. We set $L:=\sup_\Omega f$ and assume $F(0,0,x)= 0$, which amounts to a change of $f(x)$.

We will also consider the larger set of functions which satisfy in the viscosity sense the set of inequalities
\begin{equation}\label{princineq} (S^*)\;
\left\{\begin{aligned}
M^+_{\lambda, \Lambda}(D^2 u) + K |\grad u|   &\geq -L  \\
M^-_{\lambda, \Lambda}(D^2 u) - K |\grad u|   &\leq L
\end{aligned}
\right.\qquad \mbox{in }\; \Omega.
\end{equation}

A natural concept of weak solution for fully nonlinear equations is that of a viscosity solution (standard references on the general theory of viscosity solutions include \cite{UG}, \cite{caffarelli1995fully}). We denote the above problems with $(S)$ and $(S^*)$ in order to use the same notation as in \cite{caffarelli1995fully}.

Viscosity solutions  are a priori only continuous functions, so it is clearly a fundamental problem to understand whether and when a viscosity solution has some smoothness. A {\it regularity result} starts from a merely continuous solution and shows that the function is in fact more regular (for example, belongs to $C^{\alpha}$, $C^{1,\alpha}$ or $C^{2,\alpha}$). This must not be confused with an {\it a priori estimate}, in which one assumes from the beginning that the solution is classical, and only proves an estimate on the size of some norm. The a priori estimates are technically easier to prove because one can make computations with derivatives  of the solution  without worrying about their existence and continuity. A regularity result is practically always accompanied by an a priori estimate, but not necessarily the other way around.

Boundary a priori estimates for solutions to fully nonlinear elliptic equations were first proved by Krylov in  \cite{Kr1}, who thus upgraded his and Evans' interior $C^{2,\alpha}$-estimates for convex fully nonlinear equations to global estimates. More references will be given below.

In this paper we prove some {\it boundary regularity results} for viscosity solutions, in situations when these solutions do not have the same regularity in the interior of the domain. We stress that all the estimates we prove are known (at least to the experts) if the solution is a priori assumed to be globally smooth. Due to this, one may expect that the corresponding results for viscosity solutions can be obtained by direct extension to viscosity solutions of the known techniques. It turns out however that some difficulties specific to viscosity solutions arise, and workarounds become necessary. These will be discussed in more detail below.

Before stating the main theorems, we make several simple observations on the relation between $(S)$ and $(S^*)$. Obviously if $u$ satisfies $(S)$ then it satisfies $(S^*)$. The converse is true  if $u$ is a classical solution of $(S^*)$, in the sense that  there exists a linear operator $F$ (depending on $u$ and not necessarily continuous in $x$) satisfying (H1) such that $u$ is a solution of $(S)$. However, in general viscosity solutions of $(S^*)$ are not solutions of a uniformly elliptic equation in the form $(S)$. An important  observation is that under (H1) each partial derivative of a $C^1$-smooth solution of $F(D^2u,Du)=0$ is a viscosity solution of $(S^*)$, by the stability properties of viscosity solutions with respect to uniform convergence.

Our first theorem concerns the boundary $C^{1,\alpha}$-regularity of solutions of $(S^*)$. In the sequel we assume that $0\in \partial \Omega$, and denote $\Omega^+_R= \Omega\cap B_R$, $\Omega^0_R = \partial \Omega \cap B_R$, where $B_R=B_R(0)$ is the ball centered at $0$ with radius $R$.

\begin{thm} \label{t:boundaryharnack}
Suppose (H1) holds, $\Omega$ is a $C^{2}$-domain and $u$ is a viscosity solution to \re{princineq} such that the restriction $g=u|_{\partial \Omega}\in C^{1,\overline \alpha}(\Omega^0_1)$, for some $\overline \alpha>0$.
Then there exists a function $G \in C^{\alpha}(\Omega^0_{1/2},\R^d)$, the "gradient" of $u$ on $\partial\Omega$,  such that
\begin{equation} \label{e:targetC1a0}
\|G\|_{C^{\alpha}(\Omega^0_{1/2})} \leq CW,
\end{equation}
and for every $x  \in \Omega^+_1$ and every $x_0\in \Omega^0_{1/2}$ we have
\begin{equation} \label{e:targetC1a}  |u(x) - u(x_0)-  G(x_0) \cdot (x-x_0)  | \leq C W |x-x_0|^{1+\alpha},
\end{equation}
where
$$
W:=\|u\|_{L^\infty(\Omega^+_1)}+L+\|g\|_{C^{1+\alpha}(\Omega^0_1)}.
$$
Here $\alpha=\alpha(d, \lambda, \Lambda)\in (0,\overline \alpha)$;  $C$  depends on $d$, $\lambda$, $\Lambda$, $K$, and the maximal curvature of $\Omega$.
\end{thm}

The second theorem concerns the boundary $C^{2,\alpha}$-regularity of solutions of $(S)$. We need to assume that $F$ is H\"older continuous in $x$, in the following sense
\begin{itemize}
 \item[(H2)] there exist $\overline \alpha, \overline C>0$ such that for all $M\in \S_d ,p\in \R^d$, $x,y\in \overline \Omega$,
\[|F(M,p,x) - F(M,p,y)| \leq \overline C |x-y|^{\overline \alpha} (|M|+|p|).\]
\end{itemize}

Note that (H1)-(H2)  imply that the solutions of $(S)$ have H\"older continuous gradients in~$\overline \Omega$, see Theorem \ref{t:regulglobal} below.

\begin{thm} \label{t:boundaryfullynonlinear}
Suppose (H1)-(H2) hold, $\Omega$ is a $C^{2,\overline \alpha}$-domain, and $f\in C^{\overline \alpha}(\Omega)$. Let $u$ be a viscosity solution to \re{princeq} such that the restriction $g=u|_{\partial \Omega}\in C^{2,\overline \alpha}(\Omega^0_1)$.
Then there exists a function  $H\in C^\alpha(\Omega^0_{1/2},\R^{d \times d})$, the  "Hessian" of $u$ on $\partial\Omega$, such that
\begin{equation} \label{e:targetC2a1}F(H(x_0), D u (x_0), x_0) = f(x_0) \quad\mbox{ for each } \; x_0 \in \Omega^0_{1/2},\qquad \|H\|_{C^\alpha(\Omega^0_{1/2})} \leq  CW,
\end{equation}
and  for every $x \in \Omega^+_1$  and every $x_0 \in \Omega^0_{1/2}$ we have
\begin{equation} \label{e:targetC2a} |u(x) - u(x_0) - D u(x_0)  \cdot (x-x_0) -\frac{1}{2}
H(x_0) (x-x_0) \cdot (x-x_0) | \leq C W |x-x_0|^{2+\alpha} ,
\end{equation}
 where
 $$
 W:=\|u\|_{L^\infty(\Omega^+_1)} + \|f\|_{C^{\overline \alpha}(\Omega^+_1)} + \|g\|_{C^{2,\overline \alpha}(\Omega^0_{1})}.
 $$
Here $\alpha=\alpha(d, \lambda, \Lambda, \overline \alpha)>0$; $C$  depends on  $d$, $\lambda$, $\Lambda$, $K$, $\overline \alpha$, $\overline C$ and the $C^{2,\overline \alpha}$ regularity of $\partial \Omega$.
\end{thm}

The solutions in the above theorems do not have the same regularity in the interior of the domain as on the boundary. Specifically,  solutions of $(S^*)$ are in general only H\"older continuous in $\Omega$ and solutions of $(S)$ have only H\"older continuous gradients in $\Omega$; and these cannot be improved, at least if $d\geq 5$.
Indeed, it was proved by Nadirashvili and Vladut \cite{NV} that for each $\beta>0$ there exists a operator $F=F(M)$ which satisfies (H1) and can even be taken rotationally invariant and smooth, such that $F(D^2u)=0$ has a $(1+\beta)$-homogeneous solution in $B_1$. The derivatives of $u$ are then solutions of $(S^*)$ which do not belong to $C^\beta(B_1)$.

Note in these counterexamples the singularity of the solution occurs in the center of the ball, i.e. far from the boundary. By combining Theorem \ref{t:boundaryfullynonlinear} with a "regularity under smallness" result due to Savin, we can show that solutions of $(S)$ are $C^{2,\alpha}$-smooth in a whole neighbourhood of a $C^{2,\overline \alpha}$-smooth level set, provided $F(M,p,x)$ is $C^1$ in the $M$-variable.

 \begin{thm} \label{t:boundaryneighbourhood}
 Suppose (H1)-(H2) hold, $\Omega$ is a $C^{2,\overline \alpha}$ domain, and $f\in C^{\overline \alpha}(\Omega)$. Suppose in addition that $F(M,p,x)$ is continuously differentiable in $M$. Let $u$ be a viscosity solution to \re{princeq} such that the restriction $g=u|_{\partial \Omega}\in C^{2,\overline \alpha}(\Omega^0_1)$. Then there exist $\alpha, \delta>0$  such that $u\in C^{2,\alpha}(\Omega_\delta)$, where
 $\Omega_\delta = \{x\in \Omega \;:\; \mathrm{dist}(x,\partial \Omega)<\delta\}$. Here $\alpha=\alpha(d, \lambda, \Lambda, \overline \alpha)>0$; $\delta$ depends on  $d$, $\lambda$, $\Lambda$, $K$, $\overline \alpha$, $\overline C$, $\partial \Omega$, and a modulus of continuity of $D_MF$ on $\mathcal{B}_{C_0}\times \overline\Omega$, where $\mathcal{B}_{C_0}$ is a ball in $\S_d\times \R^d$ with radius $C_0$ depending on $d$, $\lambda$, $\Lambda$, $K$, $\overline \alpha$, $\overline C$, $\partial \Omega$.
 \end{thm}

\noindent {\bf Acknowledgement. } The result of Theorem \ref{t:boundaryneighbourhood} was suggested by Nikolai Nadirashvili to the first author after his talk in a conference in Paris, in June 2011.

 Another application of Theorem \ref{t:boundaryfullynonlinear} is contained in \cite{sirakovsilvestre} where we used this theorem to deduce Serrin-like symmetry results for fully nonlinear overdetermined problems, without making regularity assumptions on the solution.

\subsection{Discussion of difficulties and more context}

In general one expects a regularity result to hold whenever an a priori estimate exists. This is in particular the case for {\it global estimates} in the presence of a uniqueness result for viscosity solutions, since then one can use the a priori estimate and the standard continuity method to link the fully nonlinear equation to the Laplace equation, and deduce the existence of a solution in the space where the a priori estimate is proven. Sometimes it is also possible to approximate the equation by more treatable  equations, but in general it is difficult to approximate a fully nonlinear elliptic PDE with some equation that retains its main properties (for results in that direction we refer to \cite{caffarelli2010smooth} and \cite{krylov2012regularity}).

Furthermore, translating the proof of an estimate from classical to viscosity solutions has some obvious difficulties. Every time a derivative of the solution would be written down and used for an estimate, an alternative argument is needed. In many cases, there is some more or less standard procedure for extending a proof from classical to viscosity solutions. In a few cases however, there are some special difficulties that make this task much more complicated. The most extreme example is probably the uniqueness of solutions to second order elliptic fully nonlinear equations. While the comparison principle is obvious for classical solutions, it is an important result in the theory of viscosity solutions (see \cite{jensen1988maximum} and \cite{jensen1988uniqueness}). Another fundamental difference is that classical solutions of $(S^*)$ are solutions of $(S)$ for some $F$, while viscosity solutions are not, in general.

As we noted above, boundary a priori estimates for non-divergence form elliptic operators were first proved by  Krylov, in sections 4-5 of \cite{Kr1}.
It was already observed in that paper that boundary $C^{2,\alpha}$-estimates do not require convexity of the operator.
 A fundamental role in the proof of these estimates is played by an "improvement of oscillation" estimate close to the boundary for solutions of linear equations with zero boundary condition. Shortly after Krylov's work appeared simplifications of the  proof of this estimate, due to Safonov (see \cite{Saf1}) and Caffarelli (unpublished work, to our knowledge referred to for the first time in \cite{kazdan1985prescribing}). The most easily accessible source for Krylov's improvement of oscillation estimate is Theorem 9.31 in \cite{gilbarg2001elliptic}, where the proof from \cite{kazdan1985prescribing} is given.  In that book  the result is stated for  strong solutions, and only in the setting of a flat boundary and zero boundary data. It turns out that the proof in \cite{gilbarg2001elliptic}, as well as the proof in \cite{Saf1},  can be extended to viscosity solutions in the $S^*$ class in arbitrary domains with zero boundary data. However, a difficulty  arises, somewhat unexpectedly, when trying to extend the same result to arbitrary $C^{1,\alpha}$-smooth boundary data, due to the lack of "splitting" in the set of solutions of $(S^*)$. Let us describe this interesting open problem.

 \noindent {\it Open problem}. Let $u$ be a solution to $(S^*)$. Is it true that $u= v+w$, where $v$ solves $(S^*)$ and $v=0$  on $\partial \Omega$ and  $w$ is a solution to $(S^*)$ with $L=0$ ? More simply, say $u$ is a solution of $M^+(D^2 u)\ge f(x) \ge M^-(D^2 u)$ in $\Omega$, is it true that we have the splitting $u=v+w$ where  $v$ satisfies the same inequalities and vanishes on $\partial \Omega$,  while $M^+(D^2 u)\ge 0 \ge M^-(D^2 u)$ in $\Omega$?

   Note that this statement would provide a direct argument, based on the maximum principle, which reduces a general $C^{1,\alpha}$ regularity result to one for functions that vanish on the boundary. Such an argument is described for instance in the proof of Proposition 2.2 in \cite{milakis2006regularity}.

   Note also that the answer to the above question is clearly affirmative if $u$ is a classical solution. Furthermore,  using such splitting is not needed if the boundary data is supposed to be $C^{2}$-smooth, since then one can just remove a $C^{2}$-smooth function from the solution and obtain a new solution which vanishes on the boundary. These two remarks probably explain why this open problem has not been observed before.

We circumvent the lack of splitting by using a Caffarelli-type iteration argument, in which the iteration step is insured by the use of an implicit bound provided by global H\"older estimates, see Lemma \ref{l:boundaryperturbation} and Theorem \ref{t:c1alpha-dirichlet}.

Another example of a difficulty exclusive to viscosity solutions appears in the proof of Theorem \ref{t:boundaryfullynonlinear}.
In Lemma \ref{l:lemmaboundaryfullynonlinear} we prove that if the boundary is flat, a solution to an autonomous fully nonlinear equation which vanishes on the boundary has a second order expansion there, with the corresponding $C^{2,\alpha}$-bound.
This lemma can be proved by essentially applying the $C^{1,\alpha}$ estimates of Theorem \ref{t:boundaryharnack} to the normal derivative $\partial_d u$ -- a well-known idea (note $\partial_d u$ does not vanish on the boundary). Previously we have to prove that $\partial_d u$ is $C^{1,\alpha}$ on the boundary. The known way to do that is to apply Theorem \ref{t:boundaryharnack} (with $g$=0) to the tangential derivatives $\partial_i u$ for $i=1,\dots,d-1$. This implies that $\partial_i u$ is $C^{1,\alpha}$ and in particular $\partial_{d}\partial_i u$ is $C^\alpha$ on the flat boundary.
At this point one would want to imply that $\partial_d u$ is $C^{1,\alpha}$ on the boundary, which is obvious for a classical solution, since $\partial_{d}\partial_i  u = \partial_{i}\partial_d  u$. But for viscosity solutions these second derivatives do not have the classical meaning, and cannot be defined in any way for points that are away from the boundary.

It is worth mentioning that we have an alternative proof of Lemma \ref{l:lemmaboundaryfullynonlinear} and Theorem \ref{t:boundaryfullynonlinear} which only uses Theorem \ref{t:boundaryharnack} in the particular case $g=0$. This proof is based on  a direct barrier construction, and does not apply Theorem \ref{t:boundaryharnack} to $\partial_d u$.

Another particularity in the proof of Theorem \ref{t:boundaryfullynonlinear} appears in the passage from the specific case considered in Lemma \ref{l:lemmaboundaryfullynonlinear} to the general equation \re{princeq}. The perturbation argument that we use is based on an approximation lemma,
Lemma \ref{l:perturbation}, which appears to be new. This lemma says that two solutions of different equations which are close to each other differ by at most a precise algebraic upper bound. This is a version of Lemma 7.9 from \cite{caffarelli1995fully} which does not require the equation to have $C^{1,1}$ estimates.
\medskip

Finally, let us give some more context on regularity results for viscosity solutions of fully nonlinear equations. Caffarelli proved in his breakthrough paper \cite{CafAnnMath} that the Alexandrov-Bakelman-Pucci and Harnack inequalities are valid for viscosity solutions of $F(D^2u,x)=f(x)$, and deduced that these solutions are locally in $C^{1,\alpha}$ (resp. in  $C^{2,\alpha}$), in the presence of a priori bounds in $C^{1,\alpha}$ (resp. $C^{2,\alpha}$) for the solutions of $F(D^2u,0)=0$. A complete account of the theory of the latter equation is given in the book \cite{caffarelli1995fully}. For generalizations to equations with measurable coefficients and the so-called $L^p$-viscosity solutions we refer to \cite{Sw}, \cite{CraKoc}.  Global regularity results and estimates for viscosity solutions can be found in the appendix of  \cite{milakis2006regularity} as well as in \cite{Wi}. Combining the results from all these works we obtain the following global results, which we state for the reader's convenience and completeness.

 \begin{thm} \label{t:regulglobal} (a) Assume (H1)-(H2). If $u$ is a viscosity solution of \re{princeq} in the bounded $C^2$-domain $\Omega$, and $g=u|_{\partial \Omega}\in C^{1,\overline\alpha}(\partial\Omega)$ then  $u\in C^{1,\alpha}(\Omega)$, with a norm bounded by the quantity $W$ from Theorem \ref{t:boundaryharnack} (with $\Omega_1^+$ replaced by $\Omega$ and $\Omega_1^0$ replaced by $\partial \Omega$).

 (b) If in addition the equation $F(D^2u,0,0)=0$ admits global a priori bounds in $C^{2,\overline\alpha}(\Omega)$,  and $g=u|_{\partial \Omega}\in C^{2,\overline\alpha}(\partial\Omega)$ then  $u\in C^{2,\alpha}(\Omega)$, with a norm bounded by the quantity $W$ from Theorem \ref{t:boundaryfullynonlinear} (with $\Omega_1^+$ replaced by $\Omega$ and $\Omega_1^0$ replaced by $\partial \Omega$).
 \end{thm}

 We recall that the first assumption in Theorem \ref{t:regulglobal} (b)  is verified if  $F(M,0,0)$ is convex in $M$. The convexity assumption can be removed in some cases, see \cite{CabCaf}, but not in general.

 Theorem \ref{t:regulglobal} can be compared to Theorems \ref{t:boundaryharnack}-\ref{t:boundaryneighbourhood} from the introduction. In these theorems we assume much less on the solution but prove only boundary  regularity (and, as we already noted, interior regularity does not hold).

 We also observe that it is well-known how to put together boundary regularity results such as the ones proved in Theorems \ref{t:boundaryharnack}-\ref{t:boundaryfullynonlinear} and  interior regularity results, in order to deduce global statements. A simple procedure of this sort can be found for instance in Propositions~2.3 and 2.4 in \cite{milakis2006regularity}.


\section{Preliminaries}

In the sequel we denote with
$B_1^+$ the half ball $\{x=(x',x_d) \in \R^d : |x| < 1 \text{ and } x_d >0 \}.$
The bottom boundary of the half ball is $B_1^0=\{x=(x',0) \in \R^d : |x'| < 1\}.$

We  recall that we can always  perform a change of variables to flatten the boundary. Indeed, if $\Omega$ is a $C^{2}$ domain (resp. $C^{2,\overline\alpha}$ domain) then, for any point $x \in \partial \Omega$, there is a $C^2$ (resp. $C^{2,\overline\alpha}$) diffeomorphism $\varphi$ which maps a neighborhood of $x$ in $\Omega$ to the upper half ball $B_1^+$. The following proposition recalls the equation satisfied by $u \circ \varphi^{-1}$.

\begin{prop} \label{p:changeofvars}
1. If $u$ is a solution to $F(D^2 u, Du, x) = 0$ in $\Omega$, then $v(x) = u(\varphi^{-1}(x))$ is a solution in $B_1^+$ to
\[ F(D\varphi^t(\varphi^{-1}(x)) D^2 v(x) D\varphi(\varphi^{-1}(x)) + Dv(x) D^2 \varphi(\varphi^{-1} x) , Dv(x) D \varphi(\varphi^{-1}(x)),\varphi^{-1}(x)) = 0\]
If we denote with $\tilde F (D^2 v(x), Dv(x), x)$ the operator in the left-hand side of this equality and $ F$ satisfies (H1) and/or (H2), then $\tilde F$ satisfies (H1) and/or (H2), with possibly modified constants $K,L,\overline C$, depending only on the $C^2$ (resp. $C^{2,\overline\alpha}$) norm of $\varphi$.

2. If $u$ is a solution to \re{princineq}, then $v(x) = u(\varphi^{-1}(x))$ is a solution to \re{princineq}, with possibly modified constants $K,L$, depending only on the  $C^2$  norm of $\varphi$.
\end{prop}

\begin{proof} This follows from a straightforward computation and use of the definition of a viscosity solution.
\end{proof}

We  observe that gradient terms and explicit $x$-dependence are unavoidable after the change of variables. That is why it would not simplify the problem to consider equations without gradient terms or independent of $x$ in the theorems in the introduction. \medskip


 \begin{prop}[interior Harnack inequality] \label{p:harnack}
Let $u$ be a nonnegative solution of \eqref{princineq} in $B_1^+$. Then for each compact subset $\Sigma$ of $B_1^+$ there exists a constant $C$ depending on $d,\lambda,\Lambda, K$, and $\Sigma$ such that
\[ \sup_{\Sigma} u \leq C (\inf_\Sigma u + L).\]
\end{prop}

\begin{proof}
This is a well-known result, see for instance Theorem 4.3 in \cite{caffarelli1995fully} or \cite{Wa}.
\end{proof}

In the following we set $e = (0,\dots,0,1/2)$.

\begin{prop}[Harnack inequality up to the boundary] \label{p:harnackglobal}
Let $u$ be a nonnegative solution of \eqref{princineq} in $B_1^+$ which vanishes on $B_1^0$. Then
\[ \sup_{B_{1/2}^+} u \leq C (u(e)+L).\]
The constant $C$ depends on $d,\lambda,\Lambda$, and $K$.
\end{prop}

\begin{proof}
This is Theorem 1.3 in  \cite{BCN}. In that paper only classical solutions and linear equations were considered; however exactly the same proof applies in our situation, since the proof in \cite{BCN} uses only the comparison principle.
\end{proof}

\begin{prop}[Lipschitz estimate]
\label{p:lipschitz}
Let $u$ be a solution of \eqref{princineq} in $B_1^+$ which vanishes on $B_1^0$. Then
\[ |u(x)| \leq C (u(e)+L) x_d \qquad \text{in } B_{1/2}^+.\]
The constant $C$ depends on $d,\lambda,\Lambda$, and $K$.
\end{prop}

\begin{proof} This is Lemma 2.1 in \cite{BCN}.
\end{proof}

Next, we observe that after flattening the domain we can zoom into a neighborhood of a point on $B_1^0$ (which we will always assume to be the origin),  and assume that the lower order terms in the equation are as small as we like. That is, we can set $u_r(x)= u(rx)$ and observe that the function $u_r$ satisfies
\begin{equation} \label{e:mainwitheps}
\begin{aligned}
M^+(D^2 u_r) + rK |\grad u_r|  &\geq-r^2L \; &\text{ in } B_1^+, \\
M^-(D^2 u_r) - rK |\grad u_r|   &\leq r^2L \; &\text{ in } B_1^+, \\
\end{aligned}
\end{equation}
which in particular means that we can assume, by fixing some small $r$, that in \re{princineq} we have
\begin{equation}\label{generic}
 \max\{K,L\}\leq \eps_0,
 \end{equation}
  for any initially fixed positive constant $\eps_0$. We insist that \re{generic} is generic in a neighborhood of any given point on the boundary.

\begin{prop}[Hopf principle]
\label{p:hopf}
Let $u$ be a nonnegative solution of \eqref{princineq} in $B_1^+$ which vanishes on $B_1^0$. There exists $\eps_0>0$ such that if $|K|\le \eps_0$, then
\[ u(x) \geq c_0 \left( u(e) - C_0 L \right) x_d \qquad \text{in } B_{1/2}^+.\]
The constants $\eps_0$, $c_0$ and $C_0$ depend on $d,\lambda,\Lambda$ only.
\end{prop}
\begin{proof}
This proposition is a quantitative version of Hopf's lemma in terms of extremal equations with nontrivial right-hand sides.

All constants $c,C$ with varying indices that appear below depend on $d,\lambda,\Lambda$ only. Observe that if $u(e) \le C_0L$ then we have nothing to prove. So in what follows we assume $u(e) > C_0L$ (the constant $C_0$ will be determined below).

We can assume $\eps_0\le 1$. From the interior Harnack inequality, Proposition \ref{p:harnack}, we know that for some $c_1,C_1>0$
\[ u \geq c_1 u(e) - C_1 L \qquad \text{in } B_{99/100} \cap \{x_d > 1/16\}.\]
We assume $C_0$ is chosen so that $C_0>C_1/c_1$.

Fix $x \in B_{1/4}^+$, $x = (x',x_d)$. Set $z_0 = (x',1/4)$. We define the following barrier function
\[ \Psi(y) = (c_1 u(e)-C_1 L) \left( \frac{|y-z_0|^{-p} - (1/4)^{-p}}{(1/8)^{-p} - (1/4)^{-p}}\right) + \frac{L}{\lambda d} (|y-z_0|^2 - 1/16) ,\]
where $p = 2p^*$ and $p^*=\frac \Lambda \lambda (d-1) - 1$ is the usual power such that the minimal Pucci operator vanishes when evaluated at the Hessian of $|y|^{-p^*}$, $y\not=0$. Then $\mathcal{M}_{\lambda, \Lambda}^-(D^2|y-z_0|^{-p}) \ge c(p)>0$ in $B_{1/4}(z_0) \setminus B_{1/8}(z_0)$, and the function $\Psi$  satisfies the inequalities
\begin{align*}
 \mathcal{M}_{\lambda, \Lambda}^-(D^2 \Psi) \geq    c_2(c_1 u(e)-C_1 L) +2L\geq 2L\qquad \text{ in } B_{1/4}(z_0) \setminus B_{1/8}(z_0),\\
|\grad \Psi| \leq C_2(c_1 u(e)-C_1 L) +C_3L\qquad \text{ in } B_{1/4}(z_0) \setminus B_{1/8}(z_0),\\
\Psi \leq c_1 u(e)- C_1 L \leq u \qquad \text{ on } \partial B_{1/8}(z_0),\\
\Psi = 0 \leq u  \qquad \text{ on } \partial B_{1/4}(z_0).
\end{align*}
 Therefore if
 $\eps_0$ is small enough (smaller than $c_2/(2C_2)$, $1/(2C_3)$), in the annulus $B_{1/4}(z_0) \setminus B_{1/8}(z_0)$ we have $\frac 12 \mathcal{M}_{\lambda, \Lambda}^-(D^2\Psi) \geq K |\grad \Psi|$ and
\[ \mathcal{M}_{\lambda, \Lambda}^-(D^2 \Psi) - K |\grad \Psi| \geq \frac 12 \mathcal{M}_{\lambda, \Lambda}^-(D^2\Psi)\geq L\geq \mathcal{M}_{\lambda, \Lambda}^-(D^2 u) - K |\grad u|.\]
From the comparison principle $u\geq \Psi $ in $B_{1/4}(z_0) \setminus B_{1/8}(z_0)$.

Finally, observe that
$$
\frac{\partial\Psi}{\partial x_d}(x^\prime,0)\geq c_4(c_1 u(e)-C_1 L) - C_4L\geq \frac{c_4}{2}(c_1 u(e)-C_1 L),
$$
where the last inequality is ensured by using $L\leq u(e)/C_0$ and by taking $C_0$ sufficiently large.

The proof is thus finished.
\end{proof}

Finally, we recall the following Krylov-Safonov global H\"older estimate for solutions of~$(S^*)$.

\begin{prop}[global H\"older estimate]
\label{p:holder}
Let $u$ be a solution of \eqref{princineq} in $B_1^+$. There exist positive constants $\alpha_0$, $\rho_0$ and $C$ depending on $d,\lambda,\Lambda,K$, and $L$, such that for all $\rho\in (0,\rho_0)$ and any ball $B_\rho(x)$, $x\in \overline{B_1^+}$, we have
$$
\osc_{B_\rho(x)\cap B_1^+}u\le C\left( \rho^{\alpha_0} + \osc_{B_{\sqrt{\rho}}(x)\cap \partial B_1^+} u\right).
$$
\end{prop}

\begin{proof} This follows from Theorem 2 in \cite{Si}. Observe that theorem is stated for solutions of fully nonlinear equations, however its proof is given for solutions of $(S^*)$ (see also the remark on page 603 of that paper).
\end{proof}

\section{Boundary $C^{1,\alpha}$-regularity for the class $S^*$} 


\begin{lemma} \label{l:flatboundaryharnack}
Let $u$ be a solution of \eqref{princineq} in $B_1^+$ which vanishes on $B_1^0$, and  $\|u\|_{L^\infty(B_1^+)} \leq 1$. Then we can find $A \in \R$ (representing the "normal derivative" of $u$ at the origin) such that for all $x \in B_1^+$,
\begin{equation}\label{lemloc1} |u(x) - A x_d| \leq C(1+L) |x|^{1+\alpha} \qquad \mbox{and }\quad|A| \leq C(1+L).
\end{equation}
The positive constant  $\alpha$ depends on $d,\lambda,\Lambda$ only, and $C=C(d,\lambda,\Lambda, K)$.
\end{lemma}

\begin{proof}
Replacing $u$ by $u/(1+L)$, we can assume that $L\le 1$.
We will construct an increasing sequence $V_k$ and a decreasing sequence $U_k$ so that for $r_k = 2^{-k}$, $k\ge 1$, we have
\begin{equation} \label{e:sandwitch}  V_k x_d \leq u(x) \leq U_k x_d \qquad \text{in } B_{r_k},
\end{equation}
and also
\begin{equation} \label{e:slopevariation}  U_k - V_k \leq M r_k^{\alpha},
\end{equation}
for constants $\alpha$ and $M$ which depend on the right quantities and will be determined later. The statement of the lemma easily follows from this construction by taking $A = \lim_{k \to \infty} V_k = \lim_{k \to \infty} U_k$, since for each $x \in B_{1/2}^+$ we can choose $k$ so that $r_{k+1}<|x|\le r_k$, and then \re{e:sandwitch} and \re{e:slopevariation} imply
$$
u(x) - Ax_d\le (U_k-A)x_d \le Mr_k^\alpha x_d= 2^\alpha M r_{k+1}^\alpha x_d\le 2^\alpha M |x|^{1+\alpha},
$$
and similarly $u(x) - Ax_d \ge -2^\alpha M |x|^{1+\alpha}$. If $x\in B_{1}^+\setminus B_{1/2}^+$, then \re{lemloc1} is obvious.

As a first step in the construction of the sequences $V_k$ and $U_k$,
 we obtain $V_1$ and $U_1$ from the Lipschitz estimate, Proposition \ref{p:lipschitz}. In this case we can take
$U_1 =  2\bar C$ and $V_1 = -2\bar C$, where $\bar C$ is the constant $C$ from Proposition \ref{p:lipschitz} (recall $|u|\le 1$ and $L\le1$). At this point we fix the constant $M$ as follows:
$$
M= 4\bar C r_{k_0}^{-\alpha},\quad\mbox{ where } k_0 \mbox{ is fixed so that } \quad r_{k_0}K\le \eps_0\mbox{ and } r_{k_0}^{1/2}(1+2\bar C)<\eps_1,
$$
where $\eps_0$ is the constant from Proposition \ref{p:hopf} and $\eps_1\in (0,1)$ will be chosen below.

Hence we can take $V_1 = V_2 = \dots =V_{k_0}$ and $U_1 = U_2 = \dots = U_{k_0}$, and we only need to construct $V_k$ and $U_k$ satisfying \eqref{e:sandwitch} and \eqref{e:slopevariation} for $k>k_0$.

Assume  that we have constructed all members of the sequences $V_j$ and $U_j$  up to the level $j=k$ ($k\geq k_0$ for the reason explained above). Let us now construct $V_{k+1}$ and $U_{k+1}$.

Since $V_1 \leq V_k \leq U_k \leq U_1$, we know that $|V_k|$ and $|U_k|$ are bounded by $2\bar C$.
Note that if $U_k-V_k \le M r_{k+1}^\alpha$, then we can take $V_{k+1} = V_k$ and $U_{k+1} = U_k$. So we can assume that we have $$U_k - V_k > M r_{k+1}^\alpha.$$

Set $e_k = (0,\dots,0,r_{k+1})$. From \eqref{e:sandwitch} we get
\[ V_k r_{k+1} \leq u(e_k) \leq U_k r_{k+1}.\]
We now distinguish two cases, either $u(e_k) \geq r_{k+1}(V_k+U_k)/2 $ or not. Let us first assume the former.

We introduce the following rescaled function
\[ v_k(x) = r_k^{-1-\alpha} (u(r_k x) - V_k r_k x_d),\quad x\in B_1,\]
that is, $u(x) = V_kx_d+ r_k^{1+\alpha}v_k(x/r_k)$, $x\in B_{r_k}$.

Since $V_k x_d \leq u(x) \leq U_k x_d$ in $B_{r_k}^+$, we have by \eqref{e:slopevariation}
\[ 0 \leq v_k \leq \frac{U_k-V_k}{r_k^\alpha} x_d \leq M x_d \ \text{ in }B_1^+.\]
Moreover, $v_k$ satisfies the following scaled version of \eqref{princineq}
\[
\begin{aligned}
M^+(D^2 v_k ) + r_k K |\grad v_k | + r_k^{1-\alpha} L (1 + V_k) &\geq 0 \qquad \text{ in } B_1^+, \\
M^-(D^2 v_k ) - r_k K |\grad v_k | - r_k^{1-\alpha} L (1 + V_k) &\leq 0 \qquad \text{ in } B_1^+, \\
u &= 0  \qquad \text{ on } B_1^0.
\end{aligned}
\]

The constant $\alpha$ will be chosen small enough, so we can assume $\alpha<1/2$. By the choice of $k_0$ and $k>k_0$ we have $r_k^{1-\alpha} L (1 + V_k)\le r_{k_0}^{1/2}(1+2\bar C)\le \eps_1<1$, and  $r_kK\le\eps_0$. Therefore we can apply Proposition \ref{p:hopf} and obtain
\beeq\label{loc2} v_k(x) \geq  c_0 (v_k(e_0) -C_0 \eps_1) x_d  \ \text{ in } B_{1/2}^+. \eeq

Recalling that $r_k=2^{-k}$, $u(e_k) \geq r_{k+1}(V_k+U_k)/2 $ and $U_k - V_k > M r_{k+1}^\alpha$, we see that
\[ v_k(e_0) = r_k^{-1-\alpha} (u(e_k) - V_k r_{k+1}) \geq \frac{r_k^{-\alpha}}{2}  \frac{U_k-V_k}{2} \geq  M 2^{-\alpha-1} \geq \frac{M}{4}.\]

Now we choose $\eps_1=M/(8C_0)$, to obtain from \re{loc2}
\[ v_k(x) \geq c_0 \frac M 8 x_d  \ \text{ in } B_{1/2}^+. \]
 In terms of the original scale, this means that
\[ u(x) \geq (V_k + c_0\frac M 8 r_k^\alpha ) x_d \ \text{ in } B_{r_{k+1}}^+.\]

By the induction hypothesis $U_k-V_k \leq M r_k^\alpha $,  hence
 \[ u(x) \geq (V_k + \frac{c_0}{8} (U_k - V_k) ) x_d \ \text{ in } B_{r_{k+1}}^+.\]
We now choose $U_{k+1} = U_k$ and $V_{k+1} = V_k + \frac{c_0}{8} (U_k - V_k)$. Finally, the constant $\alpha$ is chosen so that $2^{-\alpha} = (1-c_0/8)$. In this way we have
\[ U_{k+1} - V_{k+1} = (1-c_0/8) (U_k-V_k) = 2^{-\alpha} (U_k-V_k) \leq M (r_k/2)^{\alpha} =Mr_{k+1}^\alpha,\]
and we finish the iterative step.

In the case that $u(e_k) < r_{k+1}(V_k+U_k)/2 $ we proceed in a similar way, by using the scaled function
\[ v_k(x) = r_k^{-1-\alpha} (U_k r_k x_d - u(r_k x)),\]
to obtain \eqref{e:sandwitch} with $V_{k+1} = V_k$ and $U_{k+1} = U_k - \frac{c_0}{8} (U_k-V_k)$.
\end{proof}

\begin{thm} \label{t:flatboundarysstar}
Let $u$ be a solution of \eqref{princineq} in $B_1^+$ which vanishes on $B_1^0$. Then there exist $\alpha>0$ and a function $A\in C^{\alpha} (B^0_{1/2})$ such that for every $x_0\in B_{1/2}^0$,  $x \in B_1^+$ we have
\begin{equation} \label{e:fbh} \|A\|_{C^\alpha(B_{1/2}^0)} \leq CW, \qquad\mbox{and}\qquad |u(x) - A(x_0) x_d| \leq C W|x-x_0|^{1+\alpha},
\end{equation}
where
$$
W:= \|u\|_{L^\infty(B_1^+)} + L.
$$
Here $\alpha=\alpha(d,\lambda,\Lambda)$, and $C=C(d,\lambda,\Lambda, K)$.
\end{thm}

\begin{proof} Replacing $u$ by $u/W$ we can assume that
$\|u\|_{L^\infty(B_1^+)} \leq 1$ and $L\le 1$.
For each $x_0\in B^0_{1/2}$, we know that there is a constant $A(x_0)$ for which the second inequality in \eqref{e:fbh} holds, and $|A(x_0)| \leq C$. This follows from an application of Lemma \ref{l:flatboundaryharnack} appropriately translated to~$x_0$.

We now have to prove that $A(x)$ is H\"older continuous on $B^0_{1/2}$, with bounded norm. Let $x_1, x_2 \in B_{1/2}^0$ and $r = 2|x_1 -x_2|$. We assume without loss of generality that $r<1/4$. Fix a point  $y \in B_r^+(x_1) \cap B_r^+(x_2)$ with $y_d > r/4$. We have
\begin{align*}
|u(y) - A(x_1) y_d| \leq C |x_1-y|^{1+\alpha} \leq C r^{1+\alpha} = C |x_1 - x_2|^{1+\alpha}, \\
|u(y) - A(x_2) y_d| \leq C |x_1-y|^{1+\alpha} \leq C r^{1+\alpha} = C |x_1 - x_2|^{1+\alpha}.
\end{align*}
Then
\[ \begin{aligned}
|A(x_1) - A(x_2)| &\leq 4r^{-1} |(A(x_1) - A(x_2)) y_d| \\
&\leq 4r^{-1} |A(x_1) y_d - u(y)| + 4r^{-1} |u(y) - A(x_2) y_d|  \\
&\leq C |x_1-x_2|^{\alpha}.
\end{aligned}\]

\end{proof}

\begin{remark} Of course if $u\in C^1$ in a neighbourhood of $B_1^0$ then $A= Du|_{B_{1/2}^0}$. Recall however that there exist functions which satisfy \re{princineq} and are not $C^1$ in the interior of $B_1^+$.
\end{remark}


In Theorem \ref{t:flatboundarysstar} we proved a boundary gradient H\"older estimate for functions which satisfy $(S^*)$, and vanish on the boundary. We will now extend this to arbitrary $C^{1,\alpha}$-boundary data.

We first prove the following lemma.
\begin{lemma} \label{l:boundaryperturbation}
There exists $\gamma_1>0$ such that for every $\gamma\in (0,\gamma_1)$ we can find $\delta>0$  such that if $u$ is a viscosity solution of \re{princineq} in $B_1^+$ with
\[
\|u\|_{L^\infty(B_1^+)}\le 1\qquad \mbox{and}\qquad \|u\|_{L^\infty(B_1^0)}\le \delta,
\]
 then there exist $A \in \R$ such that
\beeq\label{loc3}
 |A| \leq C_1\qquad \mbox{and}\qquad|u(x) - Ax_d| \leq \gamma^{1+\alpha_1} \ \text{ for all }x \in B_\gamma^+.
 \eeq
The positive constants $\gamma_1, \alpha_1,C_1$ are such that $\alpha_1=\alpha_1(d,\lambda,\Lambda)$, and $\gamma_1, C_1$ depend only on $d$, $\lambda$, $\Lambda$, $K$ (but not on $\gamma$ or $\delta$).
\end{lemma}

\begin{proof}
 We can take $\alpha_1$ to be any positive number smaller than the exponent $\alpha$ from Theorem~\ref{t:flatboundarysstar}. We choose $C_1$ to be the constant $C$ from that theorem.

In Theorem \ref{t:flatboundarysstar} we proved that if $\delta=0$ then we can get a constant $A$ (bounded by $C_1$) such that
\[ |u(x) - A x_d| \leq C_1 |x|^{1+\alpha} \text{ for all } x \in B_1^+. \]
In particular, if we choose $\gamma_1$ so small  that $C_1 \gamma_1^{\alpha-\alpha_1} < 1/2$, we  have, if $\delta=0$,
\beeq\label{loc44} |u(x) - A x_d| \leq \frac 12 \gamma^{1+\alpha_1} \text{ for all } \gamma\in (0,\gamma_1),\;x \in B_\gamma^+. \eeq

Now, let us assume that the result we want to prove is false for the choice of $\gamma_1$, $\alpha_1$ and $C_1$ that we already made. This means that there exists $\gamma\in (0,\gamma_1)$ and  sequences  $u_k\in C(\overline{B_1^+})$ and $\delta_k \to 0$ such that each $u_k$ is a solution of \re{princineq}, with
\[
\|u_k\|_{L^\infty(B_1^+)}\le 1\qquad \mbox{and}\qquad \|u_k\|_{L^\infty(B_1^0)}\le \delta_k,
\]
and for each $A\in (-C_1,C_1)$ the second inequality in \re{loc3} is false for $u_k$ in $B_\gamma^+$.

We now apply the global estimate contained in Proposition \ref{p:holder}, and deduce that for each $\varepsilon>0$ there exist $\delta>0$ and $N$ such that $x,y\in \overline{B_{3/4}^+}$, $|x-y|<\delta$, and $k\ge N$ imply $|u_k(x)-u_k(y)|<\varepsilon$. This is enough to apply the Arzela-Ascoli theorem (or more precisely its proof), and conclude that we can extract a subsequence of $u_k$ which converges uniformly in $B_{3/4}^+$. Let $u_\infty$ be the limit of this subsequence. By the stability properties of viscosity solutions this limit function $u_\infty$ satisfies \re{princineq}  in $B_{3/4}^+$ and vanishes on   $B_{3/4}^0$.

Therefore \re{loc44} holds for $u_\infty$, there exists a bounded constant $A$, $|A|\le C_1$, such that
\[ |u_\infty(x) - A x_d| \leq \frac 12 \gamma^{1+\alpha_1} \text{ for all } x \in B_\gamma. \]
But $u_k\to u_\infty$ uniformly in $B_{2/3}^+\supseteq B_\gamma^+$. In particular $|u_k - u_\infty| \leq \gamma^{1+\alpha_1}/2$ for $k$ sufficiently large. Thus
\[ |u_k(x) - A x_d| \leq \gamma^{1+\alpha_1} \text{ for all } x \in B_\gamma, \]
and we arrive to a contradiction.
\end{proof}

\begin{thm} \label{t:c1alpha-dirichlet}
Let $u$ be a viscosity solution to \re{princineq} in $B_1^+$ such that the restriction of $u$ to the flat boundary $g=u|_{B_1^0}\in C^{1,\overline\alpha}(B_1^0)$, for some $\overline\alpha>0$.
 Then there exists a function $A \in C^{\alpha}(B_{1/2}^0)$ such that for all $x=(x',x_d) \in B_1^+$ and all $x_0=(x'_0,0)\in B_{1/2}^0$,
\begin{equation} \label{e:targetC1b}  |u(x) - \grad_{x'} g(x_0)\cdot (x'-x'_0)  - A(x_0)x_d| \leq C(\|u\|_{L^\infty(B_1^+)}+L+\|g\|_{C^{1+\alpha}(B_1^0)}) |x-x_0|^{1+\alpha},
\end{equation}
\begin{equation} \label{e:targetC1b0}
\|A\|_{C^{\alpha}(B_{1/2}^0)} \leq C(\|u\|_{L^\infty(B_1^+)}+L+\|g\|_{C^{1+\alpha}(B_1^0)}).
\end{equation}
As usual, $\alpha=\alpha(d,\lambda,\Lambda)\in (0,\overline\alpha)$, and $C$ depends on $d,\lambda,\Lambda,K$.
\end{thm}

\begin{proof}
Repeating the argument in the proof of Theorem \ref{t:flatboundarysstar}, we see it is enough to prove the result with $x_0=0$, that is,   there exist $A \in \R$ such that for all $x \in B_1^+$,
\begin{equation} \label{e:targetC1c}  |u(x) - \grad_{x'} g(0) \cdot x'  - Ax_d| \leq C(\|u\|_{L^\infty(B_1^+)}+L+\|g\|_{C^{1+\alpha}(B_1^0)}) |x|^{1+\alpha},
\end{equation}
for some universal $C$, and $|A| \leq C(\|u\|_{L^\infty(B_1^+)}+L+\|g\|_{C^{1+\alpha}(B_1^0)})$.

 Again, we are going to build an iteration process accounting for the difference at diadic scales between $u$ and an approximate solution.

By subtracting a suitable plane at the
origin (adding to $L$ the supremum of $|u|+|\grad g|$, if necessary), we can suppose that $u(0) = g(0) = 0$ and $\grad_{x'}
g(0) = 0$.

Set $$M = 2\norm{u}_{L^\infty(B_1^+)} + \frac{1}{\delta}\norm{g}_{C^{1,\alpha}(B_1^0)} + L.$$ Here $\delta$ is the constant from Lemma \ref{l:boundaryperturbation} with  a value of $\gamma$ which will be specified below.

We will  show that there are constants $\alpha >0$ (small), $\gamma>0$ (small) and $C_1>0$ (large) to be chosen below (depending only on $d$, $\lambda$, $\Lambda$ and $K$), such that we can construct a sequence of real numbers $a_k$ with
\begin{align} \label{dtar1}
\osc_{B_{\gamma^k}^+} (u(x) - a_k \cdot x_d) &\leq M r_k^{1 + \alpha} \\ \label{dbforA1} \abs{a_{k+1} - a_k} &\leq C_1
M r_k^\alpha,
\end{align}
where we have set $r_k=\gamma^k$.

We choose $a_0=0$, hence \eqref{dtar1} holds for $k=0$. We will construct the other values of $a_k$ inductively.
Let us assume that we already have a
sequence $a_k$ so that \eqref{dtar1} holds for $k=0,1,\dots,k_0$; we
have to show that there is a real number $a_{k_0+1}$ such that
\eqref{dtar1} holds for $k = k_0+1$.

Let (we write $k$ instead of $k_0$)
\[ u_k(x) = M^{-1} r_k^{-(1+\alpha)} [ u(r_k x) - a_k r_k x_d]. \]
This scaling means precisely that the \re{dtar1} is equivalent to  $\osc_{B_1^+} u_k \leq 1$. In addition, it is easy to see that \re{princineq} transforms into
\begin{align*}
M^+(D^2 u_k) + K r_k |D u_k| + K M^{-1} r_k^{1-\alpha} |a_k| + LM^{-1} r_k^{1-\alpha} \geq 0 \qquad &\text{ in } B_1^+, \\
M^-(D^2 u_k) - K r_k |D u_k| - K M^{-1} r_k^{1-\alpha} |a_k| -  LM^{-1}r_k^{1-\alpha} \leq 0 \qquad &\text{ in } B_1^+,
\end{align*}

Since $M \geq L$ and $\gamma < 1$, we have that the last term in these inequalities $L\gamma^{k(1-\alpha)} M^{-1} \leq 1$.
Next, note that
$$
|a_k| \leq \sum_{k=0}^{k-1} |a_{j+1} - a_j|\le C_1 M\sum_{k=0}^{\infty}(\gamma^\alpha)^k=\frac{C_1 M}{1-\gamma^\alpha},
$$
by using $a_0=0$ and the inductive hypothesis $\abs{a_{j+1} - a_j} \leq C_1 M \gamma^{k \alpha}$ for all $j<k$. Hence the third terms in the above differential inequalities satisfy, for all $k\ge1$
$$
 K M^{-1} r_k^{1-\alpha} |a_k| \leq \frac{KC_1\gamma^{1-\alpha}}{1-\gamma^\alpha}.
 $$
At this point we choose $\gamma\in (0,\gamma_1)$ such that
\beeq\label{loc4}
\frac{KC_1\gamma^{1-\alpha}}{1-\gamma^\alpha}<1\eeq
and deduce that $u_k$ satisfies \re{princineq} in $B_1^+$, with $L=2$.

Further, on the flat boundary we clearly  have
$$
u_k(x) =  M^{-1} r_k^{-(1+\alpha)} g(r_k x)  \qquad \text{ on } B_1^0.
$$
Since $M \geq \|g\|_{C^{1,\alpha}} / \delta$ and $g(0) = |\grad g(0)| = 0$, this implies \[\|u_k \|_{L^\infty(B_1^0)} \leq \delta.\]

Therefore we can apply
 Lemma \ref{l:boundaryperturbation} to $u_k$, and obtain that there are $C_1$ and $\alpha>0$ (this is where $C_1$ and $\alpha$ are chosen), as well as a constant $\tilde a_k$ such that $|\tilde a_k| \leq C_1$, and
\begin{equation}\label{loc15}
 |u_k(x) - \tilde a_k x_d| \leq \gamma^{1+\alpha} \text{ in } B_\gamma^+.
 \end{equation}
Note that in Lemma \ref{l:boundaryperturbation}, we can choose $\gamma$ arbitrarily small without affecting the choice of constants $\alpha$ and $C_1$, but modifying $\delta$ accordingly. So, we fix $\gamma>0$ and $\delta>0$ so that both Lemma \ref{l:boundaryperturbation} and \re{loc4} are satisfied.

We  set $a_{k+1} = a_k + M r_k^\alpha \tilde a_k$. Recall that $u(x) = a_k x_d + r_k^{1+\alpha} u_k(x/r_k)$ if $x\in B_{r_k}$. Therefore for all $x\in B_{r_k}$ we have
 $$
 u(x) - a_{k+1} x_d = Mr_k^{1+\alpha}\left( u_k(x/r_k) -\tilde a_k x_d/r_k\right).
 $$
 The last quantity is smaller than $M(r_k\gamma)^{1+\alpha}=Mr_{k+1}^{1+\alpha}$ if $x\in B_{r_{k+1}}$, by \re{loc15}. This finishes the inductive construction.

Let $A = \lim_{k \to \infty} a_k$. We claim that
\beeq\label{loc7}
\abs{u(x) - A  x_d} \leq C M \abs{x}^{1+\alpha}.
\eeq
Indeed, from \eqref{dtar1}, \eqref{dbforA1} and $\abs{a_k - A}\leq \sum_{j=k}^\infty \abs{a_j - a_{j+1}}$ we get
\begin{align}
\osc_{B_{r_k}} (u(x) - A  x_d) &\leq
\osc_{B_{r_k}} (u(x) - a_k  x_d) + \gamma^k
\abs{a_k - A} \\ &\leq M \gamma^{k  (1 + \alpha)}
+ C_1 M \gamma^k \sum_{j=k}^\infty \gamma^{\alpha j}  \\ &\leq M
\gamma^{k  (1 + \alpha)} + C_1 M \gamma^{(1+\alpha)k}
\frac{1}{1-\gamma^\alpha} \\ &\leq C M \gamma^{k \cdot (1 +
\alpha)}
\end{align}
We easily obtain \eqref{loc7} by taking $k$ such that $\gamma^{k+1} < |x| \leq \gamma^k$ and appying the last inequality.
\end{proof}

Theorem \ref{t:boundaryharnack} is a direct consequence of Theorem \ref{t:c1alpha-dirichlet}, taking $G(x_0) = (D_{x^\prime}g(x_0), A(x_0))$.

\section{$C^{2,\alpha}$ regularity on the boundary for fully nonlinear equations} 

We will first prove Theorem \ref{t:boundaryfullynonlinear} in the particular case of an autonomous equation and a solution which vanishes on a flat part of the boundary.  The general case will be obtained later by an iterative perturbative procedure.

\begin{lemma} \label{l:lemmaboundaryfullynonlinear}
Let $u$ be a viscosity solution to the equation
\[ \begin{aligned}
F(D^2 u, Du) &= 0 \text{ in } B_1^+,\\
u &= 0 \text{ on } B_1^0,
\end{aligned}
\]
and $F$ satisfies (H1). Then there is a H\"older continuous function $H:B_{1/2}^0 \to \R^{d \times d}$ such that $F(H, Du)=0$ on $B_{1/2}^0$ and for every  $x \in B_1^+$, $x_0 \in B_{1/2}^0$,
\begin{equation} \label{e:t1}  |u(x) - D u(x_0)\cdot (x-x_0) - \frac 12 H(x_0) (x-x_0) \cdot (x-x_0)| \leq C |x-x_0|^{2+\alpha} \|u\|_{L^\infty(B_1^+)} .
\end{equation}
 In addition
\[ \begin{aligned}
\|H\|_{C^\alpha(B_{1/2}^0)} \leq  C \|u\|_{L^\infty(B_1^+)}.
 \end{aligned}
 \]
Here $\alpha=\alpha(d,\lambda,\Lambda)$ and $C$ depends on $d,\lambda,\Lambda$,  and $K$.
\end{lemma}

\begin{proof} Recall that $u\in C^{1,\alpha}(B_{3/4}^0)$ for some $\alpha>0$, see Theorem \ref{t:regulglobal}.
Without loss of generality, we assume that $\|u\|_{L^\infty(B_1^+)} = 1$ (if not, set $a=\|u\|_{L^\infty(B_1^+)}$ and replace $u$ by $u/a$ and $F(M,p)$ by $(1/a)F(aM,ap)$).

For each $i\in \{1,\ldots, d-1\}$, by (H1)  the incremental quotient $v_h (x)=\frac{1}{h}( u(x + he_i) - u(x))$ satisfies in $B_{1-h}^+$ the inequalities of Theorem \ref{t:boundaryharnack}, with $L=0$. Since viscosity solutions are stable with respect to uniform convergence, the partial derivative $u_i=\partial_i u$ is also a solution of the same inequalities.
Since $u \equiv 0$ on $B_1^0$ and $B_1^0$ is flat, we have $\partial_i u \equiv 0$ on $B_1^0$.

 Thus, by applying Theorem \ref{t:boundaryharnack} to $\partial_i u$, for each $i = 1, \dots, d-1$ and $x_0 \in B_{3/4}^0$ we obtain a quantity $G_i(x_0)$ which is a H\"older continuous function on $B_{3/4}^0$, and
\begin{equation} \label{e:Mid}  |\partial_i u(x) - G_i(x_0) \cdot (x-x_0)| \leq C |x-x_0|^{1+\alpha}, \qquad x\in B_1^+.
\end{equation}

We now define $$H_{ij}(x_0) = 0\;\mbox{ for }i,j = 1, \dots, d-1,\qquad\mbox{and}\qquad H_{di}(x_0):=G_i(x_0)\;\mbox{ for }i = 1, \dots, d-1.$$  Note that by definition $H_{di}(x_0)$ represents $\partial_d \partial_i u(x_0)$. Since  $u$ is not necessarily a $C^2$ function in a neighborhood of $B_1^0$, we cannot commute the partial derivatives to conclude that $\partial_i \partial_d u$ is H\"older continuous on $B_{3/4}^0$ (these quantities are not even partial derivatives in the classical sense).

We need to justify that \eqref{e:Mid} implies that $u_d=\partial_d u$ is $C^{1,\alpha}$ on $B_{3/4}^0$, and that its tangential derivatives coincide with $H_{di}$. This is the content of the following claim. \medskip

\noindent {\it Claim}.  The restriction of the normal derivative $u_d$   to $B_{3/4}^0$ is a $C^{1,\alpha}$ function, and  $\partial_i \partial_d u = H_{di}$ on $B_{3/4}^0$, for each $i = 1, \dots, d-1$.
\medskip

\noindent {\it Proof}.  Without loss of generality, we prove  that $u_d$ is $C^{1,\alpha}$ at the origin.
Let $\tau$ be a tangential unit vector, say $\tau=e_i$ for some $i=1,\dots,d-1$. Given two small positive numbers $h$ and $k$, we are going to estimate the difference $u(ke_d + h\tau) - u(ke_d)$ in two different ways.

 On one hand,
\begin{align*}
u(k e_d + h\tau) - u(ke_d) &= h \ u_\tau(k e_d + \xi \tau) && \text{by the MVT, for some $\xi \in (0,h)$,}\\
&\leq h \left( k H_{di}(\xi \tau) + Ck^{1+\alpha} \right) && \text{using \eqref{e:Mid}}, \\
&\leq h k H_{di}(0) + C k h^{1+\alpha} + Chk^{1+\alpha}, && \text{using that $H_{di}$ is in $C^\alpha$.}
\end{align*}

On the other hand, we can also estimate that difference by using the mean value theorem with respect to the normal derivative. For some $\xi_1, \xi_2 \in (0,k)$ we have
\begin{align*}
u(k e_d + h\tau) - u(ke_d) &= k u_d(\xi_1 e_d+h\tau) - k u_d(\xi_2 e_d) && \text{using that $u\in C^1$ and  $u=0$ on $B_1^0$,}\\
&\geq k u_d(h\tau) - k u_d(0) - C k^{1+\alpha}&& \text{using that $u_d \in C^\alpha$.}
\end{align*}
Combining the two estimates above, and dividing by $k$, we obtain
\[ u_d(h\tau) - u_d(0) \leq Ck^\alpha + h H_{di}(0) + C h^{1+\alpha} + Ck^\alpha h.\]
Since the left hand side of this inequality is independent of $k$, we can let $k\to 0$, to obtain
\[ u_d(h\tau) - u_d(0) \leq h H_{di}(0) + C h^{1+\alpha}.\]
The inequality $u_d(h\tau) - u_d(0) \geq h H_{di}(0) - C h^{1+\alpha}$ follows analogously (switching the inequalities and the sign of all error terms above). Therefore
\[ |u_d(h\tau) - u_d(0) - h H_{di}(0)| \leq C h^{1+\alpha}.\]
This means literally that $u_d \in C^{1,\alpha}(B_{3/4}^0)$ and $\partial_\tau u_d = H_{di}$ on $B_{3/4}^0$. The claim is proved.
\medskip

We thus define $H_{id}(x_0) := H_{di}(x_0)$, for all $x_0\in B_{3/4}^0$, and all $i = 1, \dots, d-1$.
At this point we can finish the construction of $H$.  We  define $H_{dd}(x_0)$ as the unique real number for which $F(H(x_0), Du(x_0))=0$ (recall $F(M,p)$ is strictly increasing in the matrix $M$). Since $F$ is Lipschitz and $H_{ij} \in C^\alpha$ for $i<d$ or $j<d$, it is obvious that  $H_{dd} \in C^\alpha(B_{3/4}^0)$.

It remains  to show \eqref{e:t1}. Without loss of generality, we will show that this inequality is valid for $x_0=0$.

We start by estimating $u(x',t)-u(0,t)$ for any $t>0$. In the following repeated indexes denote summation for $i = 1,\dots,d-1$.
\[ \begin{aligned}
u(x',t) - u(0,t) &= x_i \partial_i u(\xi,t) &&\text{by MVT, for some $|\xi|<|x'|$},\\
&\geq x_i \partial_i u(\xi,0) + H_{id}(\xi) x_i t - C_1 (t^{1+\alpha} |x'|) \\
&\geq H_{id}(0) x_i t - C_1 t |x'| (|x'|^{\alpha} + t^\alpha),
\end{aligned}\]
where we used $\partial_i u=0$ on $B_{1}^0$, \re{e:Mid} and the H\"older continuity of $H$ on the flat boundary.

Now, let us assume in order to arrive to a contradiction that for some $r>0$
\begin{equation}\label{contrad}
u(0,r) - u_d(0)r - \frac 12 H_{dd} r^2= \pm C_0 r^{2+\alpha},
 \end{equation} where  $C_0$ is a large constant to be chosen below. Say we have  plus sign in \re{contrad} (the minus sign is treated analogously). We construct the auxiliary function
\[\begin{aligned}
w(x) &= Du(0)\cdot x + \frac 12 H(0)x\cdot x + C_0 r^\alpha x_d^2- 2C_1r^\alpha |x|^2\\
&=u_d(0) x_d + H_{id}(0) x_i x_d + \frac 12 H_{dd}(0) x_d^2 + (C_0 - 2C_1) r^\alpha x_d^2 - 2C_1 r^\alpha |x'|^2,
\end{aligned}\]
where $C_1$ is the constant from the inequality on $u(x',t)-u(0,t)$, above.

Note that by \re{contrad}
\begin{equation}\label{loc16}
u(0,r) - w(0,r) = 2C_1r^{2+\alpha}.
\end{equation}

For $r$ sufficiently small, $w(x)$ is a subsolution of $F(D^2w,Dw)\ge0$ in the box $Q_r:=[-r,r]^{d-1} \times [0,r]$, provided $C_0$ is chosen sufficiently large. This is so because
\begin{align*}
D^2 w &= H(0) + 2 (C_0-2C_1) r^\alpha (e_d \otimes e_d) - 2 C_1 r^\alpha (e_i \otimes e_i) \\
Dw &= u_d(0) e_d + O(r)\quad\mbox{ in }Q_r,
\end{align*}
and hence (recalling that $F(H(0),u_d(0) e_d )=F(H(0),Du(0))=0$)
\[ F(D^2 w, Dw) \geq F(H(0),Du(0)) + 2r^\alpha M^- \left( (C_0-2C_1) (e_d \otimes e_d) - C_1 (e_i \otimes e_i) \right) - Cr \geq 0,\]
if $(C_0-2C_1) > \lambda (d-1) C_1 / \Lambda$ and $r\in(0,r_0)$, for some sufficiently small $r_0$.

Let \[ k := \max \{ w(0,t) - u(0,t) : t \in [0,r] \}.\]
Note that $k \geq 0$ since $w(0,0) = u(0,0) = 0$.

We will now see that $w \leq u+k$ on the boundary of $Q_r$. Indeed, on the bottom boundary $\{x_d = 0\}$, we have $w \leq 0 = u$. On the top, $\{x_d = r\}$ and $|x'|\le r$, we have, by the definition of $w$, the above estimate on $u(x',t)-u(0,t)$ and \re{loc16} that
\begin{align*}
w(x',r) & - u(x',r) - k \le \left( w(x',r)-w(0,r)\right) - \left( u(x',r) - u(0,r) \right) + \left( w(0,r) - u(0,r) \right)  \\
&\leq \left( H_{id}(0) r x_i - 2C_1 r^\alpha |x'|^2\right) + \left(-H_{id}(0) r x_i + C_1 |x'| r (r^\alpha+|x'|^\alpha) \right) - 2C_1 r^{2+\alpha}\\
&\leq C_1 \left( r |x'| ( r^\alpha + |x'|^{\alpha}) - 2 ( |x'|^2 + r^2) r^{\alpha} \right) \leq 0.
\end{align*}

On the side boundary, $|x'|=r$ and $t\in (0,r)$, we have
\begin{align*}
w(x',t) - u(x',t) - k &= \left( w(x',t)-w(0,t)\right) - \left( u(x',t) - u(0,t) \right) + \left( w(0,t) - u(0,t) - k\right) \\
&\leq \left( H_{id}(0) t x_i - 2 C_1 r^\alpha |x'|^2\right) + \left(-H_{id}(0) t x_i + C_1 |x'| t (t^\alpha+|x'|^\alpha) \right)\\
&= C_1 \left( -2 r^{2+\alpha} + t^{1+\alpha} r + t r^{1+\alpha} \right) \leq 0,
\end{align*}

By the comparison principle, $w \leq u + k$ everywhere in the box $Q_r$. Now, if $k>0$, this means that $w(0,t) = u(0,t) + k$ for some $t \in (0,r)$ -- a contradiction with the strong maximum principle. On the other hand, if $k=0$, we get a contradiction with the Hopf lemma,  since $\partial_d(u-w) = 0$ at the origin.\medskip

Thus, by translating the origin along $B_{2/3}^0$, we have proved that for any $x \in B_{2/3}^+$,
\begin{equation}\label{loc10} | u(x',x_d) - u(x',0) - u_d(x',0) x_d - \frac 12 H_{dd}(x',0) x_d^2 | \leq C x_d^{2+\alpha}.\end{equation}
 We now  use that $Du=u_d \in C^{1,\alpha}(B_{3/4}^0)$ and $H \in C^\alpha(B_{3/4}^0)$, which implies
$$
|u_d(x',0) - u_d(0) -H_{di}(0) x_i|\le C|x|^{1+\alpha},\qquad |H_{dd}(x',0) -H_{dd}(0)|\le C|x|^{\alpha}
$$
Then \re{e:t1} with $x_0=0$ follows from plugging the last two inequalities into \re{loc10}.

Lemma \ref{l:lemmaboundaryfullynonlinear} is proved. \end{proof}

On order to extend Lemma \ref{l:lemmaboundaryfullynonlinear} to the general equation \re{princeq} we will use the following approximation result.

\begin{lemma} \label{l:perturbation}
Assume (H1). Let $u$ be a solution to
\[
\begin{aligned}
F(D^2 u, Du, x) &= 0 \text{ in } B_1^+, \\
u &= 0\text{ on } B^0_1,
\end{aligned}
\]
and $\|u\|_{L^\infty(B_1^+)} \leq 1$. Let $v$ be a solution to
\begin{align*}
F(D^2 v, Dv, 0) &= 0 \text{ in } B_{3/4}^+, \\
v &= u \text{ on }\partial B_{3/4}^+.
\end{align*}
Assume also that for some $\kappa>0$
\[ |F(M , p, x) - F(M,p,0)| < \kappa (1+|p|+|M|).\]
Then there exist  $\gamma=\gamma(d,\lambda,\Lambda)>0$, and $C$ depending on $d,\lambda,\Lambda,K$, such that $$\|u-v\|_{L^\infty(B_{3/4}^+)} \leq C \kappa^\gamma.$$
\end{lemma}

\begin{proof}
By Proposition \ref{p:holder} the functions $u$ and $v$ are in $C^\alpha(B_{3/4}^+)$ for some $\alpha>0$, with $C^\alpha$-norms bounded by $C\|u\|_{L^\infty(B_1^+)}\le C$. We choose $\gamma = \alpha/2$.

Without restricting the generality (replacing, if necessary, $B_1^+$ and $B_{3/4}^+$ by $B_{r_0}^+$ and $B_{3r_0/4}^+$, for some fixed $r_0$ depending only on $d,\lambda,\Lambda,K$), we can assume that $\|v\|_{C^\alpha(B_{3/4}^+)}\le 1/2$.

By the H\"older regularity we clearly have $|u-v| \leq C \kappa^{\alpha/2}$ in $B_{3/4}^+ \setminus B_{3/4-\kappa^{1/2}}^+$, and also in $B_{3/4}^+ \cap \{x_d \leq \kappa^{1/2}\}$. We are left to prove the estimate in the remaining part of $B_{3/4}^+$, which we will call $D := B_{3/4-\kappa^{1/2}}^+ \cap \{x_d > \kappa^{1/2}\}$.

Now, for a small $\eps>0$, we consider the sup-convolution
\[v^\eps(x) = \max_{y \in \overline{B_{3/4}^+}} \left( v(y) - \frac 1 \eps |x-y|^2 \right).\]
From the elementary properties of sup-convolutions (see \cite{jensen1988maximum}, \cite{jensen1988uniqueness}, \cite[Chapter 5]{caffarelli1995fully}), we have $F(D^2 v^\eps, D v^\eps, 0) \geq 0$ in $D$ in the viscosity sense as long as we can make sure that for any $x \in D$ the maximum in the definition of $v^\eps$ is attained for some $y=x^*$ in the interior of $B_{3/4}^+$. This is true if we choose $\eps = \kappa^{1-\alpha/2}/2$. Indeed, recall that
\begin{equation}\label{loc18}
|x-x^*|\le \left(\eps \mathrm{osc}_H v\right)^{1/2},
 \end{equation}
 where $H$ is the set on which the maximum in the definition of $v^\eps$ is taken (see for instance Lemma 5.2 in \cite{caffarelli1995fully}). Since $v \in C^\alpha$ with a norm smaller than $1/2$, iterating \re{loc18} - first with $H=\bar B_{3/4}^+$, then with $H=\bar B_{3/4}^+\cap \bar B_{\eps^{1/2}}(x)$, then with $H=\bar B_{3/4}^+\cap \bar B_{\eps^{1/2+\alpha/4}}(x)$, etc - we get
 $$
 |x-x^*|\le \eps^{1/(2-\alpha)},\qquad\mbox{by using }\; \frac{1}{2-\alpha}= \frac{1}{2}+\frac{\alpha}{4}\sum_{i=0}^\infty \left(\frac{\alpha}{2}\right)^i.
 $$
Moreover, for all $x\in D$,
$$
v^\eps(x) - v(x) \le v(x^*)-v(x)\le 1/2 |x-x^*|^\alpha\le \eps^{\alpha/(2-\alpha)} \le \kappa^{\alpha/2}.
$$

The function $v^\eps$ is twice differentiable a.e. and semi-convex, with $D^2 v^\eps \geq - \frac 2 \eps I$ and $|D v^\eps| < \frac C \eps$. Let $\varphi$ be a $C^2$ function touching $v^\eps$ from above at a given point $x\in D$. Then clearly we also have that $D^2 \varphi(x) \geq -\frac 2 \eps I$ and $|D \varphi(x)| \leq \frac C \eps$. By the definition of a viscosity solution $F(D^2 \varphi(x) , D \varphi(x), 0)\ge 0$.

Fix a matrix $M$ such that $M \leq D^2 \varphi(x)$, $M^- = (D^2 \varphi(x))^-$ and $F(M,D\varphi(x),0) = 0$. Thus, we have $M^- \leq \frac 2 \eps I$ and from the ellipticity of $F$ it easily follows that $|M| \leq \frac C \eps$. Hence
\[
\begin{aligned}
F(D^2 \varphi(x) , D \varphi(x), x) &\geq F(M , D \varphi(x), x),\\
&= F(M , D \varphi(x), x) - F(M , D \varphi(x), 0), \\
&\geq -\frac C \eps \kappa = - C \kappa^{\alpha/2}.
\end{aligned}
\]

Therefore, we showed that
\begin{align*}
F(D^2 v^\eps , D v^\eps, x) \geq -C \kappa^{\alpha/2} \text{ in } D,\quad\mbox{and}\\
v^\eps \leq v +  \kappa^{\alpha/2}\leq u + C \kappa^{\alpha/2} \text{ on } \partial D.
\end{align*}
From the Alexandrov-Bakelman-Pucci inequality we get that $v \leq v^\eps \leq u + C \kappa^{\alpha/2}$ in $D$.

The inequality in the opposite direction follows similarly.
\end{proof}

\begin{remark}
If we assume that $F$ is convex or concave in the second derivative, then $F(D^2u,Du,0)=0$ would have $C^{1,1}$ solutions and we could prove Lemma \ref{l:perturbation} by using a simpler idea, as in Lemma 7.9 in \cite{caffarelli1995fully}. This lack of regularity of the solutions is compensated with the use of sup-convolutions.

For nonconvex equations, there is a weaker result in \cite{caffarelli1995fully} (Lemma 8.2) which is proved by compactness and thus does not give an algebraic expression for the upper bound of the difference between the two solutions.
\end{remark}

We are now ready to prove Theorem \ref{t:boundaryfullynonlinear} by an iterative argument which makes use of Lemma \ref{l:lemmaboundaryfullynonlinear}.

\begin{proof}[Proof of Theorem \ref{t:boundaryfullynonlinear}]

Without loss of generality, we can assume that the boundary of $\Omega$ is flat. Otherwise we can make a change of variables to flatten the boundary which  preserves the hypotheses on the equation $F$. So we assume that $u$ satisfies the equation in $B_1^+$ and equals zero on $B_1^0$. The latter is obtained by removing from $u$ a $C^{2,\alpha}$-extension of $g$ in $\Omega$.

We are going to show that the statement of Theorem \ref{t:boundaryfullynonlinear} is valid for $x_0=0$. We can assume without loss of generality that the $C^\alpha$ norm of $F$ in $B_1^+$ is less than $\eps_0$, a constant to be chosen. To achieve the latter, just replace $B_1^+$ by $B_{r_0}^+$, for some $r_0$ such that $\overline{C}r_0^{\bar\alpha}<\eps_0$, where $\overline{C}$ is the constant from (H2).

We can also assume that $\|u\|_{L^\infty(\Omega)} + \|f\|_{C^{\overline \alpha}(\Omega)}= 1$ (if not, set $a=\|u\|_{L^\infty(\Omega)}+ \|f\|_{C^{\overline \alpha}}(\Omega)$ and replace $u$ by $u/a$ and $F$ by $(1/a) F(aM,ap,x)$). By Theorem \ref{t:regulglobal} we know that the gradient $Du$ is H\"older continuous up to the boundary, so we can replace $F(M,p,x)$ by $\widetilde{F}(M,x)=F(M, Du(x),x)$ (we will write $F$ instead of $\tilde F$).

We will construct iteratively two sequences $A^k \in \R$ and $H^k \in \S_d$ such that
\begin{align}
|A^k - A^{k+1}| &\leq C r_k^{1+\alpha}, \label{e:ih1}\\
|H^k - H^{k+1}| &\leq C r_k^{\alpha }, \label{e:ih2}\\
|u(x) - A^k x_d - H^k_{ij} x_i x_j| &\leq r_k^{2+\alpha},\quad \text{if }\; |x| \leq r_k,\label{e:ih3}
\end{align}
where $r_k = \rho^k$, for some $\rho \in (0,1)$  to be determined later, depending on the right quantities.
Moreover, along the sequence, we have $H^k_{ij} = 0$ for $i,j = 1,\dots,d-1$. That is, $(Hx,x)=H_{ij}^k x_i x_j = 0$ when $x \in B_1^0$.

For $k=0$ the choice $A^k=0$ and $H^k = 0$ obviously works. Now we assume we have constructed these sequences up to certain value of $k$ and aim to find $A^{k+1}$, $H^{k+1}$.

Note that, by the induction hypothesis,
$$
|H^k|\leq \sum_{i=1}^k |H^i - H^{i-1}|\leq C \sum_{i=1}^\infty (\rho^\alpha )^{k}\leq C,
$$
and similarly for $A_k$.

Let  $P_k(x)= A^k x_d + H^k_{ij} x_i x_j$ and $u_k$ be the rescaled function
\[ u_k(x) = r_k^{-2} u(r_k x) - r_k^{-1} A_k x_d -  H^k_{ij} x_i x_j =  r_k^{-2} (u(r_kx) - P_k(r_kx)),\quad x\in B_1^+,\]
that is, $u(x) = P_k(x) + r_k^2 u_k(x/r_k)$ for $x\in B_{r_k}$.
Then we have  $|u_k|\leq r_k^\alpha$ in $B_1^+$ (by \re{e:ih3}) and
\[  F\left(  D^2 u_k + H^k,  r_k x \right) = f(r_kx)\quad \mbox{ in } \; B_1^+.\]

Let $v_k$ be the solution to the following equation
\begin{align}\label{loc14}
 F\left( D^2 v_k + H^k,  0 \right) &= f(0) \text{ in } B_{3/4}^+,\\
v_k &= u_k \text{ on } \partial B_{3/4}^+.
\end{align}

We use (H2) and apply Lemma \ref{l:perturbation}, to obtain that $|u_k - v_k| \leq C \eps_0^\gamma r_k^{\overline \alpha \gamma }$ in $B_{3/4}^+$.
We take $\alpha$ to be a positive number smaller than $\overline \alpha \gamma$.

Now, by applying Lemma \ref{l:lemmaboundaryfullynonlinear} to \re{loc14} we get that there exists $\hat \alpha>0$, $\tilde A^k$ and $\tilde H^k$ such that for all $x\in B_{3/4}^+$
\[ |v_k(x) - \tilde P_k(x)|=|v_k(x) - \tilde A^k x_d - \tilde H^k_{ij} x_i x_j| \leq C_1  \|v_k\|_{L^\infty} |x|^{2+\hat \alpha}\leq 2 C_1  r_k^\alpha |x|^{2+\hat \alpha} ,\]
where we also used that $|v_k|\le |u_k|+|u_k-v_k|\le Cr_k^\alpha$.

Here we choose $\alpha<\hat\alpha$ and  $\rho$ so that $2C_1 \rho^{\hat \alpha - \alpha} < 1/2$, thus
\[ |v_k(x) - \tilde P_k(x)| \leq r_k^\alpha\frac {\rho^{2+\alpha}} 2, \quad\mbox{ for all }\;x\in B_\rho.\]

We now choose $\eps_0$ so small that
$$
|u_k - v_k| \leq C \eps_0^\gamma r_k^{\overline \alpha \gamma }\leq r_k^\alpha \frac{\rho^{2+\alpha}}{2}\quad\mbox{ for all }\;x\in B_{3/4}.
$$

Finally, we define $P_{k+1}(x) = P_k(x) + r_k^2 \tilde P_k(x/r_k)$, in other words, $A^{k+1} = A^k + r_k \tilde A^k$ and $H^{k+1} = H^k + \tilde H^k$.
Then,  if $|x|\leq r_{k+1} = r_k/\rho$ and $y=x/r_k$ we have
$$
|u(x) - P_{k+1}(x)|=r_k^2|u_k(y)- \tilde P_k(y)|\leq r_k^2(|u_k(y)-v_k(y)| + |v_k(y)- \tilde P_k(y)|)\leq r_{k+1}^{2+\alpha}.
$$

The conditions \eqref{e:ih1} and \eqref{e:ih2} are clearly satisfied for $k+1$ since by Lemma \ref{l:lemmaboundaryfullynonlinear} and the global $C^{1,\alpha}$-estimates we have
$$|\tilde A^k|, |\tilde H^k| \leq C \|v_k\|_{L^\infty} \leq C r_k^\alpha.$$
This finishes the construction of the sequences $A^k$ and $H^k$.

Therefore we can define
$$
P(x) = \lim P_k(x) = \sum_{k=1}^\infty (P_{k+1}- P_k),
$$
since the last series is convergent. In addition, if $x\in B_{r_k}$ we have, by \eqref{e:ih1} and \eqref{e:ih2},
$$
|P(x) - P_k(x)|\le \sum_{j=k}^\infty |P_{j+1}(x)-P_j(x)|\le Cr_k^{2+\alpha}.
$$
 Thus, for each $x\in B_{3/4}^+$ we can fix $k$ such that $r_{k+1}< |x|\le r_k$ and estimate $|u(x) - P(x)|\le |u(x) - P_k(x)|+|P_k(x) - P(x)|\le C r_k^{2+\alpha}$.

Finally, we know that $v_k$ converges uniformly to zero in $B_{3/4}^+$, so \re{loc14} implies that $F(H, 0) = f(0)$, where $H=D^2P$. It only remains to show that the symmetric matrix function $H(x_0)$ which we thus constructed for all $x_0\in B_{1/2}^0$ is H\"older continuous on $B_{1/2}^0$. This is simple to get, since $F(H(x_0), x_0) = f(x_0)$, $F(M,x)$ and $f(x)$ are H\"older continuous in $x$, $F$ is Lipschitz and uniformly elliptic in~$M$, and $H_{ij}(x_0)=0$ for $i,j=1,\dots,d-1$.

The proof of Theorem \ref{t:boundaryfullynonlinear} is finished.
\end{proof}
\medskip

\noindent {\it Proof of Theorem \ref{t:boundaryneighbourhood}}. As in the previous proof, we can assume that $g=0$, the boundary is flat, and we can write $F(M,x)$ instead of $F(M,p,x)$. From Theorem \ref{t:boundaryfullynonlinear} we know that at any point $x_0 \in \partial \Omega$ there exists a second order polynomial $P=P_{x_0}$, which is H\"older continuous in $x_0$ and such that $|u(x) - P(x)| \leq C |x-x_0|^{2+\alpha}$ for some $\alpha > 0$.

Let $x \in \Omega_\delta$. From the definition of $\Omega_\delta$, there exists a point $x_0 \in \partial \Omega$ such that $|x-x_0| = \dist(x,\partial \Omega) < \delta$. Let $r = |x-x_0|/2$. We have that $B_r(x) \subset \Omega$ and $|u(x) - P(x)| \leq C_1 r^{2+\alpha}$ in $B_r(x)$.

In \cite{Sa}, Ovidiu Savin proved that  solutions with sufficiently small oscillation are $C^{2,\alpha}$-smooth. We will use the extensions of this result given in \cite[Proposition 4.1]{ASilS} and \cite[Theorem 1.2]{PT}, which say that if $F(M,x)$ is $C^1$ in $M$ and $|u(x) - P(x)| \leq \eps r^2$ in $B_r(x)$ for sufficiently small $\eps>0$, then $u \in C^{2,\alpha}(B_{r/2}(x))$ (we replace $u$ by $u-P$ and $F(M,x)$ by $F(M+D^2P,x)$). This smallness assumption is satisfied if we choose $\delta$ such that $C_1 \delta^\alpha < \eps$. Hence $u$ is $C^{2,\alpha}$-smooth in the interior of $\Omega_\delta$.

To put together this interior regularity result with the boundary result from Theorem \ref{t:boundaryfullynonlinear} we repeat the proof of Proposition 2.4 in \cite{milakis2006regularity}.
This proves that $u \in C^{2,\alpha}$ in a neighborhood of any point in $\Omega_\delta$. The rest follows by an easy covering argument.

\section*{Acknowledgments}

Luis Silvestre was partially supported by NSF grants DMS-1254332 and DMS-1065979.

\bibliographystyle{plain}
\bibliography{bdary}

\def\polhk#1{\setbox0=\hbox{#1}{\ooalign{\hidewidth
  \lower1.5ex\hbox{`}\hidewidth\crcr\unhbox0}}}
\begin{thebibliography}{10}

\bibitem{ASilS}
S.~N. Armstrong, L.~Silvestre, and C.~K. Smart.
\newblock Partial regularity of solutions of fully nonlinear, uniformly
  elliptic equations.
\newblock {\em Comm. Pure Appl. Math.}, 65(8):1169--1184, 2012.

\bibitem{BCN}
H.~Berestycki, L.~A. Caffarelli, and L.~Nirenberg.
\newblock Inequalities for second-order elliptic equations with applications to
  unbounded domains. {I}.
\newblock {\em Duke Math. J.}, 81(2):467--494, 1996.
\newblock A celebration of John F. Nash, Jr.

\bibitem{CabCaf}
X.~Cabr{\'e} and L.~A. Caffarelli.
\newblock Regularity for viscosity solutions of fully nonlinear equations
  {$F(D^2u)=0$}.
\newblock {\em Topol. Methods Nonlinear Anal.}, 6(1):31--48, 1995.

\bibitem{caffarelli2010smooth}
L.~Caffarelli and L.~Silvestre.
\newblock Smooth approximations of solutions to nonconvex fully nonlinear
  elliptic equations.
\newblock {\em Nonlinear Partial Differential Equations and Related Topics:
  Dedicated to Nina N. Uraltseva}, 64:67, 2010.

\bibitem{CafAnnMath}
L.~A. Caffarelli.
\newblock Interior a priori estimates for solutions of fully nonlinear
  equations.
\newblock {\em Ann. of Math. (2)}, 130(1):189--213, 1989.

\bibitem{caffarelli1995fully}
L.~A. Caffarelli and X.~Cabre.
\newblock {\em Fully nonlinear elliptic equations}, volume~43.
\newblock Amer Mathematical Society, 1995.

\bibitem{UG}
M.~G. Crandall, H.~Ishii, and P.-L. Lions.
\newblock User's guide to viscosity solutions of second order partial
  differential equations.
\newblock {\em Bull. Amer. Math. Soc. (N.S.)}, 27(1):1--67, 1992.

\bibitem{CraKoc}
M.~G. Crandall, M.~Kocan, and A.~{\'S}wi{\polhk{e}}ch.
\newblock {$L^p$}-theory for fully nonlinear uniformly parabolic equations.
\newblock {\em Comm. Partial Differential Equations}, 25(11-12):1997--2053,
  2000.

\bibitem{PT}
D.~dos Prazeres and E.~Teixeira.
\newblock Asymptotic regularity for flat solutions to fully nonlinear elliptic
  problems.
\newblock {\em Preprint, http://arxiv.org/abs/1302.6554}.

\bibitem{gilbarg2001elliptic}
D.~Gilbarg and N.S. Trudinger.
\newblock {\em Elliptic partial differential equations of second order}, volume
  224.
\newblock Springer Verlag, 2001.

\bibitem{jensen1988maximum}
R.~Jensen.
\newblock The maximum principle for viscosity solutions of fully nonlinear
  second order partial differential equations.
\newblock {\em Archive for Rational Mechanics and Analysis}, 101(1):1--27,
  1988.

\bibitem{jensen1988uniqueness}
R.~Jensen, P.-L. Lions, and P.~E. Souganidis.
\newblock A uniqueness result for viscosity solutions of second order fully
  nonlinear partial differential equations.
\newblock {\em Proc. Amer. Math. Soc}, 102(4):975--978, 1988.

\bibitem{kazdan1985prescribing}
J.~L. Kazdan.
\newblock Prescribing the curvature of a riemannian manifold.
\newblock Conference Board of the Mathematical Sciences, 1985.

\bibitem{Kr1}
N.~V. Krylov.
\newblock Boundedly inhomogeneous elliptic and parabolic equations in a domain.
\newblock {\em Izv. Akad. Nauk SSSR Ser. Mat.}, 47(1):75--108, 1983.

\bibitem{krylov2012regularity}
N.V. Krylov.
\newblock On regularity properties and approximations of value functions for
  stochastic differential games in domains.
\newblock {\em preprint arXiv:1207.3758, to appear in Ann. Prob.}

\bibitem{milakis2006regularity}
E.~Milakis and L.~E. Silvestre.
\newblock Regularity for fully nonlinear elliptic equations with neumann
  boundary data.
\newblock {\em Communications in Partial Differential Equations},
  31(8):1227--1252, 2006.

\bibitem{NV}
Nadirashvili N. and Vladut S.
\newblock Singular solutions of hessian elliptic equations in five dimensions.
\newblock {\em Journal de Mathématiques Pures et Appliquées},
  (http://dx.doi.org/10.1016/j.matpur.2013.03.001):in press, 2013.

\bibitem{Saf1}
M.~V. Safonov.
\newblock Smoothness near the boundary of solutions of elliptic {B}ellman
  equations.
\newblock {\em Zap. Nauchn. Sem. Leningrad. Otdel. Mat. Inst. Steklov. (LOMI)},
  147:150--154, 206, 1985.
\newblock Boundary value problems of mathematical physics and related problems
  in the theory of functions, No. 17.

\bibitem{Sa}
O.~Savin.
\newblock Small perturbation solutions for elliptic equations.
\newblock {\em Comm. Partial Differential Equations}, 32(4-6):557--578, 2007.

\bibitem{sirakovsilvestre}
L.~Silvestre and B.~Sirakov.
\newblock Overdetermined problems for fully nonlinear elliptic equations.
\newblock {\em Preprint, arxiv}.

\bibitem{Si}
B.~Sirakov.
\newblock Solvability of uniformly elliptic fully nonlinear {PDE}.
\newblock {\em Arch. Ration. Mech. Anal.}, 195(2):579--607, 2010.

\bibitem{Sw}
A.~{\'S}wiech.
\newblock {$W^{1,p}$}-interior estimates for solutions of fully nonlinear,
  uniformly elliptic equations.
\newblock {\em Adv. Differential Equations}, 2(6):1005--1027, 1997.

\bibitem{Wa}
L.~Wang.
\newblock On the regularity theory of fully nonlinear parabolic equations. {I}.
\newblock {\em Comm. Pure Appl. Math.}, 45(1):27--76, 1992.

\bibitem{Wi}
N.~Winter.
\newblock {$W^{2,p}$} and {$W^{1,p}$}-estimates at the boundary for solutions
  of fully nonlinear, uniformly elliptic equations.
\newblock {\em Z. Anal. Anwend.}, 28(2):129--164, 2009.

\end{thebibliography}
\index{Bibliography@\emph{Bibliography}}%

\end{document}